\documentclass[fullpage, 11pt]{article}

\usepackage[dvips]{graphicx}

\usepackage[english]{babel}
\usepackage{amsfonts}
\usepackage{amssymb}
\usepackage{amsmath}
\usepackage{latexsym}
\usepackage{amsthm}
\usepackage{epsfig}
\usepackage{graphicx}
\usepackage{pdfpages}
\usepackage{color}
\usepackage{colortab}
\usepackage{subfigure}
\usepackage{enumerate}
\usepackage{tikz}
\usepackage{tikz-3dplot}
\usetikzlibrary{patterns}

\newcommand{\bb}{\mathbb}

\newcommand{\conv}{\mathrm{conv}}
\newcommand{\cone}{\mathrm{cone}}

\newcommand{\R}{\bb R}

\newcommand{\Z}{\bb Z}

\newcommand{\rec}{\mathrm{rec}}
\newcommand{\st}{: \;}

\newcommand{\old}[1]{{}}

\newtheorem{prop}{Proposition}
\newtheorem{theorem}[prop]{Theorem}

\newtheorem{obs}[prop]{Observation}
\newtheorem{claim}{Claim}

\newcommand{\spcl}[1]{{\mathcal{S}(#1)}}

\newcommand{\spclt}[2]{{\mathcal{S}^{#2}(#1)}}
\newcommand{\dspcl}[2]{{\mathcal{S}_{#1}(#2)}}
\newcommand{\dspclt}[3]{{\mathcal{S}_{#1}^{#3}(#2)}}

\old{ \addtolength{\oddsidemargin}{-30pt}
\addtolength{\evensidemargin}{-30pt} \addtolength{\textwidth}{60pt}

\addtolength{\topmargin}{-22pt} \addtolength{\textheight}{32pt} }

\begin{document}
\thispagestyle{empty}

\begin{center}\LARGE Every rational polyhedron has finite split rank: new proof.
\\[\baselineskip]
\large Kanstantsin Pashkovich\footnote{Department of Combinatorics and Optimization, University of Waterloo.({\tt kanstantsin.pashkovich@gmail.com})}
\\[\baselineskip]
\end{center}
\begin{abstract}
 Split rank of a rational polyhedron is finite. The well known proof of this is based on the fact that split closure is stronger than the Chv\'{a}tal closure, and the Chv\'{a}tal rank of a rational polyhedron is finite due to the result of Chv\'{a}tal and Schrijver. In this note we provide an independent proof for the fact that every rational polyhedron has finite split rank. In principal, we construct a nonnegative potential function which decreases by at least one with ``every" second split closure unless the integer hull of the polyhedron is reached.
\end {abstract}

%\section {}
The \emph{split closure} of a polyhedron $Q\subseteq \R^n$ can be defined in the following way
$$\spcl{Q}:=\bigcap_{c\in \Z^n}\conv(Q\cap\{x\in \R^n\st cx\in \Z\}).$$
Alternatively, the split closure of $Q$ can be defined as follows
$$\bigcap_{c\in \Z^n}\bigcap_{\delta \in \Z}\conv\big( \,(Q\cap\{x\in \R^n\st cx\le \delta\})\cup (Q\cap\{x\in \R^n\st cx\ge \delta+1\})\,\big)\,,$$ because for all $c\in \R^n$ and $\bar{x}\in \R^n$, we have
$$\bar x\in \conv\big(\{x\in Q\st cx\in \Z\}\big) \iff \bar{x}\in \conv\big(\{x\in Q\st cx=\lfloor c\bar{x}\rfloor\,\text{  or  }\, cx=\lceil c\bar{x}\rceil\}\big)\,.$$
Cook, Kannan and Schrijver~\cite{CKS} proved that if  $Q$ is a rational polyhedron, then $\spcl{Q}$ is a rational polyhedron as well. Alternative proofs of this result can be found in~\cite{KGY}, \cite{SGL} and \cite{Vielma}, see also Theorem 5.10 in~\cite{CCZ}.
%(This is an unorthodox definition of the split closure, but it is the appropriate one for this paper. In Section....the equivalent "classic" definition is discussed.)
\smallskip

Clearly, for any polyhedron $Q\subseteq \R^n$ we have $\conv(Q\cap \Z^n)\subseteq \spcl{Q}\subseteq Q$.
 Furthermore if $Q$ is a rational polyhedron and $\conv(Q\cap \Z^n)\subsetneq Q$, then $ \spcl{Q}\subsetneq Q$.  So if we let $\spclt{Q}{0}:=Q$ and iteratively define  $\spclt{Q}{t}:=\spcl{\spclt{Q}{t-1}}$, then these polyhedra provide increasingly tighter outer approximations of $\conv(Q\cap \Z^n)$.
 The \emph{split rank} of a polyhedron $Q$ is the smallest $t$ for which $\spclt{Q}{t}=\conv(Q\cap \Z^n)$. In fact, we have the following theorem.

 \begin{theorem}\label{TH-splitmain} Every rational polyhedron has finite split rank.
 \end{theorem}

In this note, we provide an independent proof of Theorem \ref{TH-splitmain}. A well known proof of this theorem relies on Chv\'atal closure.  The Chv\'atal closure  of a polyhedron is a fundamental concept in integer programming and combinatorial optimization, see \cite{CCZ} Chapter 5.  
 \begin{theorem}[\cite{Chvatal},\cite{Schrijver}]\label{TH-Chrank}
	For a rational polyhedron $Q\subseteq \R^n$, $\conv(Q\cap \Z^n)$ can be obtained by iteratively taking the Ch\'atal closures a finite number of times.
\end{theorem}
Since the split closure of a polyhedron is contained in the Chv\'atal closure, the following Observation~\ref{obs-inclusion} shows that Theorem \ref{TH-Chrank} implies Theorem \ref{TH-splitmain}. 

\begin{obs}\label{obs-inclusion}
	Given polyhedra $P$, $Q$, such that $P\subseteq Q$ and $(P\cap \Z^n)= (Q\cap \Z^n)$, the split rank of $P$ is less than or equal to the split rank of $Q$.
\end{obs}

In the remaining part of the note, we present our independent proof of Theorem~\ref{TH-splitmain}.

\bigskip

 %In fact we prove (Theorem \ref{TH-Splitrestricted} below) that given a rational polyhedron $Q$, we can a priori identify a finite set of integer vectors that achieve finite convergence to $\conv(Q\cap \Z^n)$.

\section{The proof of the main result}

Throughout this section we let $Q\subseteq\R^n$ be a nonempty rational polyhedron and we let  $P:=\conv(Q\cap \Z^n)$.

\bigskip

\begin{claim}\label{LE-Cproperties}  There exists an integral system $Cx\le a$ that defines $P$ and a vector $b\ge a$ such that $Q\subseteq\{x\in \R^n\st Cx\le b\}$.
\end{claim}

\begin{proof}   It suffices to show that there exists an integral system of inequalities $Cx\le a$ that defines $P$ such that $\rec(Q)\subseteq\{x\in \R^n\st Cx\le 0\}$.

Meyer's Theorem, see Theorem 4.30 in \cite{CCZ}, states that if $Q$ is a rational polyhedron, then $P$ is a rational polyhedron and $\rec(P)=\rec(Q)$ whenever $P\ne \varnothing$. So if $P\ne \varnothing$ any minimal inequality description of $P$ satisfies the lemma. Assume now $P=\varnothing$, then let $c\in \Z^n$ be a zero vector, $a=-1$ and $b=0$, finishing the proof.
\end{proof}

\bigskip

Let $C$ be a matrix satisfying Claim \ref{LE-Cproperties}. To prove Theorem \ref{TH-splitmain} it is enough to show that for every row $c\in\Z^n$  of $C$ there exists some $t$ such that the corresponding inequality~ 
\begin{equation}\label{condition}
cx \le \alpha\quad\text{holds for all }\quad x\in\spclt{Q}{t}\,.
\end{equation}
\bigskip

Fix a row $c\in\Z^n$  of $C$. Let us assume that there is no $t$ such that~\eqref{condition} holds.

\bigskip
Consider the face $F:=\{x\in \rec(Q)\st cx=0\}$ of $\rec(Q)$. Since $F$ is a nonempty face of $\rec(Q)$, there exist $k:=n-\dim(F)$ linearly independent vectors $g_i$, $i=1,\ldots,k$ such that inequalities $g_i x \le 0$ are valid for $\rec(Q)$ and $F=\{x\in\rec(Q):g_ix=0, i=1,\ldots,k\}$. 

It is not hard to see that there is a lattice basis $w_1,\ldots, w_n\in \Z^n$ for $\Z^n$ such that $w_1,\ldots,w_k$ lie in $\cone(g_1,\ldots,g_k)$. This is due to the fact that $w_1,\ldots, w_n\in \Z^n$ form a lattice basis for $\Z^n$ if and only if the parallelepiped spanned by $w_1,\ldots,w_n$ and the origin contains no integral point in its interior\footnote{A most straightforward way to choose the desired lattice basis $w_1,\ldots, w_n$ is to take a  lattice-free parallelepiped in $\cone(g_1,\ldots,g_k)$ going through origin, and choose the neighbours of the origin in this parallelepiped as the vectors $w_1,\ldots, w_k$. It is easy to see, that $w_1$,\ldots, $w_k$ can be extended to a lattice basis for $\Z^n$.}. Due to the construction, we have that
\begin{equation}\label{dimension}
F\subseteq\{x\in \R^n\st w_i x=0, i=1,\ldots,k\}\,.
\end{equation}

\bigskip

For $i=1,\ldots,k$ define
 \begin{equation}\label{EQ-Mci}M_i:= \max\{0,\, \lceil-\min_{\{x \in Q: c x\ge \alpha\}} w_i x\rceil,\, \lceil\max_{\{x\in Q\}} w_i x\rceil\}\,.
 \end{equation}
 Before moving further let us clarify why the maximum and minimum appearing in the definition exist. Note, that $w_1$,\ldots ,$w_k$ are conic combinations of $g_1$,\ldots, $g_k$, where linear functions corresponding to $g_1$,\ldots, $g_k$ achieve maximum over $Q$. On the other hand, we may assume ${\{x \in Q: c x\ge \alpha\}}\ne \varnothing$, since otherwise~\eqref{condition} trivially holds for $Q$. Moreover,  the polyhedron $\{x \in Q: c x\ge \alpha\}$ is not only nonempty, but its recession cone equals $F$ and each of the vectors $w_1$,\ldots, $w_k$ is orthogonal to $F$ by~\eqref{dimension}, showing that linear functions corresponding to $w_1$,\ldots, $w_k$ achieve minimum over ${\{x \in Q: c x\ge \alpha\}}$. Hence, $M_i<\infty$ for every $i=1,\ldots,k$.
 
\bigskip
 For $i=1,\ldots,k$ define the directions 
 \begin{equation}\label{def_directions}
 d_0:=c\quad\text{ and }\quad d_{i}:=(2 M_i+1) d_{i-1}+w_i.
 \end{equation}
  Denote ${\mathcal D}:=\{d_0,\ldots,d_k\}$ and define
$$\dspcl{\mathcal{D}}{Q}:=\bigcap_{d\in \mathcal{\mathcal{D}}}\conv\big(Q\cap \{x\in \R^n: dx \in \Z\}\big )$$
and let $\dspclt{\mathcal{D}}{Q}{0}:=Q$, $\dspclt{\mathcal{D}}{Q}{t}:=\dspcl{\mathcal{D}}{\dspclt{\mathcal{D}}{Q}{t-1}}$.

\bigskip

Since $\mathcal{D}$ is finite, it is clear that $\dspclt{\mathcal{D}}{Q}{t}$ is a polyhedron for every $t$.
If $\dspclt{\mathcal D}{Q}{t}=\varnothing$ for some $t$ then~\eqref{condition} trivially holds, finishing the proof.
 Moreover, since $\mathcal {D}\subseteq \cone(c,g_1,\ldots,g_k)$ we can define $u^t_d:=\max_{x\in \dspclt{\mathcal D}{Q}{t}} d x<\infty$ for every $t$ and every $d\in\mathcal{D}$. 

\begin{claim}\label{two iterations}
For every $t$ and every $d\in \mathcal{D}$, we have $u^{t+1}_d\le u^{t}_d$.  Moreover, if $u^{t+1}_d< u^{t}_d$ then $u^{t+2}_d\le u^{t}_d-1$ or $u^{t+1}_d\le u^{t-1}_d-1$.
\end{claim}
\begin{proof}
The inequality $u^{t+1}_d\le u^{t}_d$ trivially holds. Moreover, if both $u^{t+1}_d$ and $u^{t}_d$ are integral, then the second statement of the claim holds as well. Let $u^{t}_{d}\not\in \Z$, then
\begin{equation*}
u^{t+1}_{d}\le \lfloor u_d^{t}\rfloor=\lceil u_d^{t}\rceil -1 \le  u^{t-1}_{d}-1\,.
\end{equation*}
The case $u^{t+1}_{d}\not\in \Z$ is analogous.
\end{proof}

\bigskip
Due to Claim~\ref{two iterations} for $d:=c$, we have either  $\lim_{t\to\infty}u^t_c=-\infty$ or there are $t_c, \gamma_c\in \Z$ such that $u^t_c=\gamma_c$ for every $t\ge t_c$. In the first case, \eqref{condition} trivially holds for some $t$. So we may assume that there are $t_c$ and $\gamma_c\in \Z$ such that $u^t_c=\gamma_c$ for every $t\ge t_c$.  If $\gamma_c\le \alpha$ then $\eqref{condition}$ holds for $t=t_c$, so we can assume $\gamma_c> \alpha$.

Let us prove that for every $d\in\mathcal{D}$ there exist $t_d$, $\gamma_d\in \Z$ such that $u^t_d=\gamma_d$ for $t\ge t_d$. To do this, it is enough to show that $\lim_{t\to\infty}u^t_d\neq-\infty$ for any $d\in\mathcal{D}$, i.e. it is enough to show that the sequence $u^t_d$ is bounded from below. Using $\dspclt{\mathcal D}{Q}{t}\cap \{x\in \R^n: cx= \gamma_c\}\neq\varnothing$, we have 

$$
u^t_d= \max_{x\in \dspclt{\mathcal D}{Q}{t}} d x \ge \max_{\substack{x\in \dspclt{\mathcal D}{Q}{t}\\ cx= \gamma_c}} d x\ge \min_{\substack{x\in \dspclt{\mathcal D}{Q}{t}\\  cx= \gamma_c}} dx\ge  \min_{\substack{x\in Q\\ cx= \gamma_c}} dx\,.
$$
Note that $\min_{\{x\in Q:\, cx= \gamma_c\}} dx\neq -\infty$ because the recession cone of ${Q\cap\{x:\, cx= \gamma_c\}}$ equals $F$, and hence is orthogonal to $d$.

Hence, we may assume that for every $d\in \mathcal{D}$ there exist $t_d$, $\gamma_d\in \Z$ such that $u^t_d=\gamma_d$ for every $t\ge t_d$. Define $t_{\mathcal{D}}:=\max\{t_d:\,d\in\mathcal{D}\}$.

\bigskip

\begin{claim}\label{claim box}
For $i=1,\ldots,k$, we have
\begin{equation}\label{box}
	\{x\in  \dspclt{\mathcal D}{Q}{t_{\mathcal{D}}+1}: d_i x =\gamma_{d_i}\}\subseteq \{x\in \R^n: d_{i-1} x =\gamma_{d_{i-1}}\}.
\end{equation}
\end{claim}

\begin{proof}
To prove that~\eqref{box} holds for $i:=j+1$, we may assume that for $i=1,\ldots, j+1$ we have
\begin{equation}\label{box_induction}
\{x\in \dspclt{\mathcal D}{Q}{t_{\mathcal{D}}+1}: d_{i-1} x =\gamma_{d_{i-1}}\}\subseteq \{x\in \R^n: c x =\gamma_c\}\,.
\end{equation}

To show~\eqref{box}, let us define $Q_{d_{i-1}}:=\conv(\{x\in \dspclt{\mathcal D}{Q}{t_{\mathcal{D}}}: d_{i-1}x\in \Z\})$. Since $\dspclt{\mathcal D}{Q}{t_{\mathcal{D}}+1}\subseteq Q_{d_{i-1}}$, to show~\eqref{box} it is enough to prove that 
\begin{equation}\label{box_induction_intermediate}
	\{x\in Q_{d_{i-1}}: d_i x =\gamma_{d_i}\}\subseteq \{x\in \R^n: d_{i-1} x =\gamma_{d_{i-1}}\}.
\end{equation}
Let us assume the contrary, i.e. there is a vertex $\bar{x}$ of $Q_{d_{i-1}}$ such that $d_i \bar{x} =\gamma_{d_i}$ and $d_{i-1} \bar{x} \le\gamma_{d_{i-1}}-1$. Take any point $x^* \in \dspclt{\mathcal D}{Q}{t_{\mathcal{D}}+1}$ such that $d_{i-1} x^*=\gamma_{d_{i-1}}$. By~\eqref{box_induction} we have $c x^* =\gamma_c \ge \alpha$. Thus,
\begin{align*}
d_i \bar{x} - d_i x^* = (2 M_i +1) d_{i-1} \bar{x} +w_i \bar{x} -(2M_i+1) d_{i-1} x^* -w_i x^*\le \\ (2 M_i +1) (\gamma_{d_{i-1}}-1)+
\max_{x\in Q} w_i x -(2 M_i+1)\gamma_{d_{i-1}}-\min_{\{x\in Q: cx \ge \alpha\}} w_i x\le\\
(2 M_i +1) (\gamma_{d_{i-1}}-1)+M_i -(2 M_i+1)\gamma_{d_{i-1}}+M_i
 \le -1\,,
\end{align*}
contradicting $d_i \bar{x} \ge d_i x^*$ and showing that~\eqref{box_induction_intermediate} holds.
\end{proof}

\bigskip

\begin{claim}
The polyhedron $Q$ contains an integral point $x$ such that $c x=\gamma_c$.
\end{claim}
\begin{proof}
 Let  $x^*$ be a vertex of $\dspclt{\mathcal D}{Q}{t_{\mathcal{D}}+1}$ such that $d_k x^*=\gamma_{d_k}$. By~Claim~\ref{claim box} we have $d x^*=\gamma_{d}\in \Z$ for every $d\in \mathcal{D}$, using~\eqref{def_directions} we get $w_i x^*\in \Z$ for $i=1,\ldots,k$. Due to $\dim(F)=n-k$ and to~\eqref{dimension} we can find $r\in F$ such that $w_i (x^*+r)\in \Z$ for every $i=k+1,\ldots,n$. Thus, $w_i (x^*+r)\in \Z$ for every $i=1,\ldots,n$ because $w_i (x^*+r)=w_i x^*\in \Z$ for every $i=1,\ldots,k$. Since $w_i$,\ldots, $w_n$ form a lattice basis for $\Z^n$, $(x^*+r)\in Q$ is an integral point and $c (x^*+r)=c x^*=\gamma_c$, finishing the proof. 
 \end{proof}
 
Thus $Q$ contains an integral point $x$ such that $c x=\gamma_c$, implying  $\gamma_c\le\alpha$. Hence~\eqref{condition} holds for $\dspclt{\mathcal D}{Q}{t_{\mathcal{D}}}$, contradiction. This finishes the proof of Theorem~\ref{TH-splitmain}. 

\section*{Acnowledgement}
I would like to thank Michele Conforti for motivating me to do this research.  I profited from many discussions with Michele Conforti both on proofs and exposition of the note.

I am also grateful for the support of the Hausdorff Institute for Mathematics in Bonn, Germany, as well as the organizers of its trimester program on Combinatorial Optimization; part of this work was done during my stay there.

This note was finalized and presented for the first time during a research stay at Institute for Discrete Mathematics, University of Bonn. I am very grateful to the Institute for Discrete Mathematics, and particularly to Jens Vygen, for hosting me in Bonn.

\bibliographystyle{plain}

\bibliography{literature}

\end {document}